\documentclass[12pt]{article}
\usepackage{graphicx}
\usepackage{amsmath,amsthm,amssymb,amsfonts,euscript,enumerate}
\usepackage{amsthm}
\usepackage{color,wrapfig}
\usepackage{pxfonts}
\usepackage{verbatim}
\newtheorem{thm}{Theorem}[section]

\newtheorem{lem}[thm]{Lemma}

\newtheorem{rmk}[thm]{Remark}

\newcommand{\Rr}{{\mathbb{R}}}

\newcommand{\1}{\partial}

\begin{document}
\title{Another proof of the existence of homothetic solitons of the inverse mean curvature flow}
\author{Shu-Yu Hsu\\
Department of Mathematics\\
National Chung Cheng University\\
168 University Road, Min-Hsiung\\
Chia-Yi 621, Taiwan, R.O.C.\\
e-mail: shuyu.sy@gmail.com}
\date{Jan 18, 2020}
\smallbreak \maketitle
\begin{abstract}
We will give a new proof of the existence of non-compact homothetic solitons of the inverse mean curvature flow (cf. \cite{DLW}) in $\mathbb{R}^n\times \mathbb{R}$, $n\ge 2$, of the form $(r,y(r))$ or $(r(y),y)$ where $r=|x|$, $x\in\mathbb{R}^n$, is the radially symmetric coordinate and $y\in \mathbb{R}$. More precisely for any $\frac{1}{n}<\lambda<\frac{1}{n-1}$ and $\mu<0$, we will give a new proof of the existence of a unique solution $r(y)\in C^2(\mu,\infty)\cap C([\mu,\infty))$  of the equation $\frac{r_{yy}(y)}{1+r_y(y)^2}=\frac{n-1}{r(y)}-\frac{1+r_y(y)^2}{\lambda(r(y)-yr_y(y))}$, $r(y)>0$, in $(\mu,\infty)$ which satisfies $r(\mu)=0$ and $r_y(\mu)=\lim_{y\searrow\mu}r_y(y)=+\infty$. We also prove that there exist constants $y_2>y_1>0$  such that $r_y(y)>0$ for any $\mu<y<y_1$, $r_y(y_1)=0$, $r_y(y)<0$ for any $y>y_1$,  $r_{yy}(y)<0$ for any $\mu<y<y_2$,  $r_{yy}(y_2)=0$ and  $r_{yy}(y)>0$ for any $y>y_2$. Moreover $\lim_{y\to +\infty}r(y)=0$ and $\lim_{y\to +\infty}yr_y(y)=0$. 

\end{abstract}

\vskip 0.2truein



\vskip 0.2truein
\setcounter{section}{0}

\section{Introduction}
\setcounter{equation}{0}
\setcounter{thm}{0}

A one-parameter family of immersions $F:M^n\times [0,T)\to\Rr^{n+1}$ of $n$-dimensional hypersurfaces in $\Rr^{n+1}$ is  said to satisfy the inverse mean curvature flow if $M_t=F_t(M^n)$, $F_t(x)=F(x,t)$,  satisfies
\begin{equation*}
\frac{\1}{\1 t}F(x,t)=-\frac{\nu(x,t)}{H(x,t)}\quad\forall x\in M^n, 0<t<T
\end{equation*}
where $H(x,t)$ and $\nu (x,t)$ are the mean curvature and unit interior normal of the surface $M_t$ at the point $F(x,t)$. This flow has been extensively studied in the compact case by C.~Gerhardt \cite{G}, G.~Huisken and T.~Ilmanen \cite{HuI4} and J.~Urbas \cite{U}, etc. C.~Gerhardt \cite{G} and J.~Urbas \cite{U} proved independently  the existence of unique solution of the inverse mean curvature flow when the initial data is a closed star-shaped hypersurface with strictly positive mean curvature. They also  proved independently that under the inverse mean curvature flow  a closed star-shaped hypersurface with strictly positive mean curvature  converges to a homothetically expanding sphere as time goes to infinity. 

For  non-starshaped initial hypersurfaces it is known that  singularities may develop (\cite{HuI1},\cite{S}) under the inverse mean curvature flow. For example Smoczyk \cite{S} proved that  such singularities can occur when the dimension of the hypersurface is equal to two and the mean curvature tends to zero somewhere during the evolution by the inverse mean curvature flow. 

In  \cite{HuI1} and  \cite{HuI3}, G.~Huisken and T.~Ilmanen used a level set approach to study the inverse mean curvature flow and proved the famous Riemannian Penrose inequality.  In \cite{BN} H.L.~Bray and A.~Neves  used the inverse mean curvature flow to prove the Poincare conjecture for $3$-manifolds with $\sigma$-invariant greater than $\mathbb{RP}^3$.
There are also study on the inverse mean curvature flow for the non-compact case recently by B.~Choi, P.~Daskalopoulos, G.~Huisken and T.~Ilmanen \cite{CD}, \cite{DH}, \cite{HuI2}, G.~Drugan, H.~Lee and G.~Wheeler \cite{DLW} and K.M.~Hui \cite{H1}, \cite{H2}, etc.

Since solutions  of geometric flows in the large time limit or near the blow-up points of the solutions usually behave like some self-similar solutions or solitons of the geometric flows, in order to understand the asymptotic properties
of the geometric flows it is important to first understand the properties of the solitons of the geometric flow. Existence of various solitons are proved recently by G.~Drugan, H.~Lee and G.~Wheeler \cite{DLW},  G.~Huisken and T.~Ilmanen \cite{HuI2}, K.M.~Hui \cite{H1}, \cite{H2}, and D.~Kim and J.~Pyo \cite{KP1}, \cite{KP2}. 

We say that a  $n$-dimensional submanifold $\Sigma$ of $\Rr^{n+1}$ with immersion $X:\Sigma\to\Rr^{n+1}$ and non-vanishing mean curvature $H$ is a homothetic soliton for the inverse mean curvature flow if there exists a constant $\lambda\ne 0$ such that
\begin{equation}\label{homothetic-soliton-equiv-defn}
<H\nu, X>=-\frac{1}{\lambda}\quad\forall X\in\Sigma.
\end{equation}
Note that if $(X,\Sigma)$ is a homothetic soliton which satisfies \eqref{homothetic-soliton-equiv-defn} for some constant $\lambda\ne 0$, then the function
\begin{equation*}
F(x,t)=e^{\lambda t}X(x)
\end{equation*}
is a solution of the inverse mean curvature flow. As observed by G.~Drugan, H.~Lee and G.~Wheeler \cite{DLW} and  K.M.~Hui \cite{H1}, \cite{H2}, if the homothetic soliton of the inverse mean curvature flow is a radially symmetric solution in $\Rr^n\times \Rr$, $n\ge 2$, of the form $(r,y(r))$ or $(r(y),y)$ where $r=|x|$, $x\in\Rr^n$, is the radially symmetric coordinate, $y\in\Rr$, then  $r(y)$ satisfies the equation
\begin{equation}\label{r-eqn}
\frac{r_{yy}(y)}{1+r_y(y)^2}=\frac{n-1}{r(y)}-\frac{1+r_y(y)^2}{\lambda(r(y)-yr_y(y))}\,\,\,,\quad r(y)>0
\end{equation}
or equivalently $y(r)$ satisfies the equation,
\begin{equation}\label{y-eqn}
y_{rr}+\frac{n-1}{r}\cdot(1+y_r^2)y_r-\frac{(1+y_r^2)^2}{\lambda(ry_r-y)}=0
\end{equation}
where $r_y(y)=\frac{dr}{dy}$, $r_{yy}(y)=\frac{d^2r}{dy^2}$ and $y_r(r)=\frac{dy}{dr}$, $y_{rr}(r)=\frac{d^2y}{dr^2}$ etc.

In this paper we will prove the following results.

\begin{thm}\label{existence-thm}
For any $n\ge 2$, $\frac{1}{n}<\lambda<\frac{1}{n-1}$, $\mu<0$,  there exists a unique solution $r(y)\in C^2(\mu,+\infty)\cap C([\mu,+\infty))$ of 
\begin{equation}\label{r-initial-value-problem}
\left\{\aligned
&\frac{r_{yy}(y)}{1+r_y(y)^2}=\frac{n-1}{r(y)}-\frac{1+r_y(y)^2}{\lambda(r(y)-yr_y(y))},\quad r(y)>0,\quad\mbox{ in }(\mu,+\infty)\\
&r(\mu)=0,r_y(\mu)=\lim_{y\searrow\mu}r_y(y)=+\infty 
\endaligned\right.
\end{equation}
which satisfies
\begin{equation}\label{r-r-derivative-expression-positive}
r(y)-yr_y(y)>0\quad\forall y>\mu.
\end{equation}
\end{thm}  

\begin{thm}\label{asymptotic-property-thm}
For any $n\ge 2$, $\frac{1}{n}<\lambda<\frac{1}{n-1}$, $\mu<0$,  let  $r(y)\in C^2(\mu,+\infty)\cap C([\mu,+\infty))$ be the  unique solution of \eqref{r-initial-value-problem} given by Theorem \ref{existence-thm} which satisfies \eqref{r-r-derivative-expression-positive}. Then  
there exist constants $y_2>y_1>0$ such that 
\begin{equation}\label{first-derivative-ineqn}
\left\{\aligned
&r_y(y)>0\quad\forall\mu<y<y_1\\
&r_y(y_1)=0\\
&r_y(y)<0\quad\forall y>y_1, 
\endaligned\right.
\end{equation}
\begin{equation}\label{second-derivative-ineqn}
\left\{\aligned
&r_{yy}(y)<0\quad\forall \mu<y<y_2\\
&r_y(y_2)=0\\
&r_{yy}(y)>0\quad\forall y>y_2, 
\endaligned\right.
\end{equation}
\begin{equation}\label{r'-limit-as-y-infty}
\lim_{y\to +\infty}yr_y(y)=0,
\end{equation}
and
\begin{equation}\label{r-limit-as-y-infty}
\lim_{y\to +\infty}r(y)=0.
\end{equation}
\end{thm}

\begin{rmk}
G.~Drugan, H.~Lee and G.~Wheeler stated the existence result Theorem \ref{existence-thm} and 
part of Theorem \ref{asymptotic-property-thm} in  \cite{DLW}. However there are no detailed proofs of these results in \cite{DLW}. On the other hand the existence result Theorem \ref{existence-thm} is proved by D.~Kim and J.~Pyo \cite{KP2} using phase plane method. In this paper we will give a different and elementary proof of these results. Note that since $y(r)$ is the inverse of $r(y)$ the result of \cite{H1} implies that Theorem \ref{existence-thm} also holds when $\lambda>\frac{1}{n-1}$.
\end{rmk}

\section{Proof of theorems}
\setcounter{equation}{0}
\setcounter{thm}{0} 

In this section we will prove Theorem \ref{existence-thm} and Theorem \ref{asymptotic-property-thm}. We first recall the following results of \cite{H1}.

\begin{lem}\label{local-existence-lem}(Lemma 2.1 of \cite{H1})
For any $n\ge 2$, $\lambda>0$ and $\mu<0$, there exists a constant $R_0>0$ such that the equation 
\begin{equation}\label{y-ivp-problem1}
\left\{\begin{aligned}
&y_{rr}+\frac{n-1}{r}\cdot(1+y_r^2)y_r-\frac{(1+y_r^2)^2}{\lambda(ry_r-y)}=0\quad\mbox{ in }(0,R_0)\\
&y(0)=\mu,\quad y_r(0)=0\end{aligned}\right.
\end{equation}
has a unique solution $y(r)\in C^1([0,R_0))\cap C^2(0,R_0)$ which satisfies
\begin{equation}\label{y-y-derivative-ineqn}
ry_r(r)-y(r)>0\quad\mbox{ in }[0,R_0).
\end{equation}
\end{lem}

\begin{lem}\label{local-y-property-lem}(cf. Lemma 2.3 and Lemma 2.4 of \cite{H1})
For any $n\ge 2$, $\lambda>0$ and $\mu<0$. Let $R_0>0$ and  $y(r)\in C^1([0,R_0))\cap C^2(0,R_0)$ be a solution of \eqref{y-ivp-problem1}  which satisfies \eqref{y-y-derivative-ineqn}.
Then
\begin{equation}\label{y-2nd-derivative-at-origin-positive}
\lim_{r\to 0}y_{rr}(r)=\frac{1}{n\lambda|\mu|}
\end{equation}
and
\begin{equation}\label{y-1st-2nd-derivatives-positive}
y_{rr}(r)>0,\quad y_r(r)>0\quad\forall 0<r<R_0
\end{equation}
and there exists a constant $\delta_1=\delta_1(R_0)$ such that
\begin{equation}\label{y-yr-lower-bd}
ry_r(r)-y(r)\ge\delta_1\quad\forall 0\le r<R_0.
\end{equation}
\end{lem}

\begin{lem}\label{local-maximal-existence-lem}
For any $n\ge 2$, $\frac{1}{n}<\lambda<\frac{1}{n-1}$ and $\mu<0$, there exist a constant $r_1\in\Rr^+$ and a unique solution $y(r)\in C^1([0,r_1))\cap C^2(0,r_1)$ of the equation 
\begin{equation}\label{y-ivp-problem2}
\left\{\begin{aligned}
&y_{rr}+\frac{n-1}{r}\cdot(1+y_r^2)y_r-\frac{(1+y_r^2)^2}{\lambda(ry_r-y)}=0\quad\mbox{ in }(0,r_1)\\
&y(0)=\mu,\quad y_r(0)=0
\end{aligned}\right.
\end{equation}
which satisfies \eqref{y-y-derivative-ineqn}, \eqref{y-2nd-derivative-at-origin-positive}
and \eqref{y-1st-2nd-derivatives-positive} with $R_0=r_1$.
Moreover 
\begin{equation}\label{y-derivative-limit1}
\lim_{r\to r_1}y_r(r)=+\infty.
\end{equation}
\end{lem}
\begin{proof}
We first observe that uniqueness of solution of \eqref{y-ivp-problem2} follows from Lemma \ref{local-existence-lem} and standard ODE theory. 
By Lemma \ref{local-existence-lem}  there exist a constant $R_0>0$ and a unique solution $y(r)\in C^1([0,R_0))\cap C^2(0,R_0)$ of the equation \eqref{y-ivp-problem1} which satisfies \eqref{y-y-derivative-ineqn}. Let $(0,r_1)$ be the maximal interval of 
existence of solution $y(r)\in C^1([0,r_1))\cap C^2(0,r_1)$ of the equation \eqref{y-eqn} which satisfies 
\begin{equation}\label{initial-y-values}
y(0)=\mu,\quad y_r(0)=0,
\end{equation}
such that \eqref{y-y-derivative-ineqn} holds for any $0<r<r_1$. By Lemma \ref{local-y-property-lem}, \eqref{y-2nd-derivative-at-origin-positive} and \eqref{y-1st-2nd-derivatives-positive}  hold with $R_0$ being replaced by $r_1$.

We claim that $r_1<\infty$. Suppose not. Then $r_1=+\infty$. Since \eqref{y-1st-2nd-derivatives-positive} now holds for any $r>0$, there exists a constant $R_1>1$
such that 
\begin{equation}\label{y-positive}
y_r(r)\ge y_r(1)>0\quad\forall r>1\quad\mbox{ and }\quad y(r)>0\quad\forall r\ge R_1.
\end{equation}
Since $\frac{1}{\lambda}>(n-1)$, by \eqref{y-ivp-problem2} and \eqref{y-positive},

\begin{equation}\label{y-rr-ineqn1}
\frac{y_{rr}}{1+y_r^2}=\frac{(n-1)yy_r+\frac{r}{\lambda}+\left(\frac{1}{\lambda}-(n-1)\right)ry_r^2}{r(ry_r-y)}
\ge\frac{n-1}{r}\cdot\frac{yy_r+r}{ry_r}
=\frac{n-1}{r}\left(\frac{y}{r}+\frac{1}{y_r}\right)\quad\forall r>R_1.
\end{equation}
Hence by  \eqref{y-positive} and \eqref{y-rr-ineqn1},
\begin{align}\label{y-rr-ineqn2}
&y_{rr}\ge\frac{n-1}{r}(1+y_r^2)\left(\frac{y}{r}+\frac{1}{y_r}\right)\ge\frac{n-1}{r}y_r\quad\forall r>R_1\notag\\
\Rightarrow\quad&(r^{1-n}y_r)_r\ge 0\quad\qquad\qquad\forall r>R_1\notag\\
\Rightarrow\quad&y_r(r)\ge R_1^{1-n}y_r(R_1)r^{n-1}\quad\forall r>R_1\notag\\
\Rightarrow\quad&y(r)\ge c_1r^n-c_2\qquad\qquad\forall r>R_1
\end{align}
where $c_1=R_1^{1-n}y_r(R_1)/n$ and $c_2=R_1y_r(R_1)/n$.
By \eqref{y-positive},  \eqref{y-rr-ineqn1} and \eqref{y-rr-ineqn2}, 
\begin{align}\label{y-rr-ineqn3}
&\frac{y_{rr}}{1+y_r^2}\ge\frac{n-1}{r^2}y(r)\ge (n-1)\left(c_1r^{n-2}-\frac{c_2}{r^2}\right)\quad\forall r>R_1\notag\\
\Rightarrow\quad&\frac{\pi}{2}>\tan^{-1}(y_r(r))\ge c_1r^{n-1}+\frac{(n-1)c_2}{r}-c_3\quad\forall r>R_1
\end{align}
where $c_3=c_1R_1^{n-1}+\frac{(n-1)c_2}{R_1}$. Since the right hand side of \eqref{y-rr-ineqn3} goes to infinity as $r\to\infty$, contradiction arises. Hence $r_1<+\infty$ and the claim follows.

By \eqref{y-1st-2nd-derivatives-positive} $y_r(r)$ is a monotone increasing function of $r\in (0,r_1)$. Hence
$a_0:=\underset{\substack{r\to r_1}}{\lim}y_r(r)$ exists and $a_0\in (0,+\infty]$. If $a_0<+\infty$, then 
\begin{equation}\label{yr-uniform-upper-bd}
0<y_r(r)\le a_0<+\infty\quad\forall 0<r<r_1.
\end{equation}
By Lemma \ref{local-y-property-lem}, there exists a constant $\delta_1=\delta_1(r_1)$ such that
\eqref{y-yr-lower-bd}  holds with $R_0$ being replaced by $r_1$. Then by \eqref{y-yr-lower-bd}, 
\eqref{yr-uniform-upper-bd}, Lemma 2.2 of \cite{H1} and the same argument as the proof of Theorem 1.1 of \cite{H1} we can extend $y(r)$ to a solution $y(r)\in C^1([0,r_1+\delta_2))\cap C^2(0,r_1+\delta_2)$ of \eqref{y-eqn} in $(0,r_1+\delta_2)$ for some constant $\delta_2>0$ that satisfy \eqref{y-y-derivative-ineqn} and
\eqref{initial-y-values}  in $[0,r_1+\delta_2)$. This contradicts the choice of $r_1$. Hence $a_0=+\infty$ and \eqref{y-derivative-limit1} follows.
\end{proof}

\begin{lem}\label{y-limit-positive-bd}
For any $n\ge 2$, $\frac{1}{n}<\lambda<\frac{1}{n-1}$ and $\mu<0$. Let $r_1\in\Rr^+$ and  $y(r)\in C^1([0,r_1))\cap C^2(0,r_1)$ be as in Lemma \ref{local-maximal-existence-lem}. Then
\begin{equation}\label{y1-defn}
y_1:=\lim_{r\to r_1}y(r)\quad\mbox{ exists and }\quad 0<y_1<+\infty.
\end{equation}
\end{lem}
\begin{proof}
By \eqref{y-1st-2nd-derivatives-positive}, $y_1:=\underset{\substack{r\to r_1}}{\lim}y(r)\in (\mu,+\infty]$ exists. We first claim that $y_1>0$. Suppose not. Then $y_1\le 0$ and 
\begin{equation}\label{y-negative}
\mu<y(r)<y_1\le 0\quad\forall 0<r<r_1.
\end{equation}
By \eqref{y-derivative-limit1} there exists a constant $R_2\in (0,r_1)$ such that
\begin{equation}\label{yr-lower-bd}
y_r(r)\ge 1\quad\forall R_2<r<r_1.
\end{equation}
Hence by \eqref{y-ivp-problem2}, \eqref{y-negative} and \eqref{yr-lower-bd},
\begin{align}\label{yrr-yr-ineqn2}
\frac{y_{rr}}{1+y_r^2}\le&-\frac{n-1}{r}y_r+\frac{1+y_r^2}{\lambda ry_r}
=\frac{\lambda^{-1}-(n-1)}{r}y_r+\frac{1}{\lambda ry_r}\quad\forall R_2<r<r_1\notag\\
\le&\frac{2\left(\lambda^{-1}-(n-1)\right)}{r_1}y_r+\frac{2}{\lambda r_1}\quad\forall R_3<r<r_1
\end{align}
where $R_3=\max (R_2,r_1/2)$. Integrating \eqref{yrr-yr-ineqn2} over $(R_3,r)$, by \eqref{y-negative},
\begin{align*}
&\tan^{-1}(y_r(r))\le\tan^{-1}(y_r(R_3))+\frac{2\left(\lambda^{-1}-(n-1)\right)}{r_1}|\mu|+\frac{2}{\lambda}\quad\forall R_3<r<r_1\notag\\
\Rightarrow\quad&y_r(r)\le\tan\left(\tan^{-1}(y_r(R_3))+\frac{2\left(\lambda^{-1}-(n-1)\right)}{r_1}|\mu|+\frac{2}{\lambda}\right)\quad\forall R_3<r<r_1
\end{align*}
which contradicts \eqref{y-derivative-limit1}. Hence $y_1>0$. 

It remains to prove that $y_1<+\infty$.
Since $r(y)$ is the inverse of $y(r)$ and $r_y(y)=1/y_r(y)$, $r_{yy}(y)=-y_{rr}(r)/y_r(r)^3$, by  
Lemma \ref{local-y-property-lem} and Lemma \ref{local-maximal-existence-lem} the function 
$r(y)\in C^2(\mu,y_1)\cap C([0,y_1))$ satisfies \eqref{r-eqn} in $(\mu,y_1)$,
\begin{equation}\label{r-initial-value}
r(\mu)=0,r_y(\mu)=\lim_{y\searrow\mu}r_y(y)=+\infty, 
\end{equation}
\begin{equation}\label{r-r-derivative-expression-positive3}
r(y)-yr_y(y)>0\quad\forall \mu<y<y_1,
\end{equation}
and
\begin{equation}\label{r''-negative2}
r_{yy}(y)<0,\quad r_y(y)>0\quad\forall \mu<y<y_1.
\end{equation}
By \eqref{r-eqn} and \eqref{r''-negative2},
\begin{align}\label{r-yy-ineqn10}
&\frac{r_{yy}(y)}{1+r_y(y)^2}\le\frac{n-1}{r(y)}-\frac{1}{\lambda r(y)}\le-\frac{c_4}{r_1}\qquad\quad\forall 0<y<y_1\notag\\
\Rightarrow\quad&\frac{c_4}{r_1}y\le\frac{c_4}{r_1}y+\tan^{-1}(r_y(y))\le \tan^{-1}(r_y(0))\quad\forall 0<y<y_1\notag\\
\Rightarrow\quad&y_1\le \frac{r_1}{c_4}\tan^{-1}(r_y(0))
\end{align}
where $c_4=\lambda^{-1}-(n-1)>0$.  Hence \eqref{y1-defn} holds and the lemma follows.
\end{proof}

Since $r(y)$ is the inverse of $y(r)$, by Lemma \ref{local-maximal-existence-lem} and Lemma \ref{y-limit-positive-bd} we have the following result.

\begin{lem}\label{local-maximal-r-existence-lem}
For any $n\ge 2$, $\frac{1}{n}<\lambda<\frac{1}{n-1}$, $\mu<0$,  there exists a constant $y_1\in\Rr^+$ and a unique solution $r(y)\in C^2(\mu,y_1)\cap C([0,y_1))$ of \eqref{r-eqn} in $(\mu,y_1)$ which satisfies
\eqref{r-initial-value}, \eqref{r-r-derivative-expression-positive3}, \eqref{r''-negative2}
and
\begin{equation}\label{r-derivative=0}
\lim_{y\to y_1}r_y(y)=0,\quad r_1:=\lim_{y\to y_1}r(y)\in (0,+\infty).
\end{equation}
\end{lem}  

\begin{lem}\label{local-r-existence-extension-lem}
For any $n\ge 2$, $\frac{1}{n}<\lambda<\frac{1}{n-1}$, $\mu<0$.  Let $y_1\in\Rr^+$ and  $r(y)$ be given by Lemma \ref{local-maximal-r-existence-lem}. Then there exists a constant $\delta_1>0$ such that $r(y)\in C^2(\mu,y_1+\delta_1)\cap C([0,y_1+\delta_1))$ is a solution of   
\eqref{r-eqn} in $(\mu,y_1+\delta_1)$ which satisfies
\eqref{r-initial-value}, 
\begin{equation}\label{r-r-derivative-expression-positive4}
r(y)-yr_y(y)>0\quad\forall \mu<y<y_1+\delta_1,
\end{equation}
\begin{equation}\label{r''-negative2-1}
r_{yy}(y)<0\quad\forall \mu<y<y_1+\delta_1
\end{equation}
and
\begin{equation}\label{first-derivative-ineqn2}
\left\{\aligned
&r_y(y)>0\quad\forall\mu<y<y_1\\
&r_y(y_1)=0\\
&r_y(y)<0\quad\forall y_1<y<y_1+\delta_1. 
\endaligned\right.
\end{equation}
\end{lem}  
\begin{proof}
By \eqref{r''-negative2},
\begin{align}\label{r-r-derivative-expression-lower-bd5}
&(r(y)-yr_y(y))_y=-yr_{yy}(y)>0\quad\forall 0<y<y_1\notag\\
\Rightarrow\quad&r_1\ge r(y)-yr_y(y)\ge r(0)>0\quad\forall 0\le y<y_1
\end{align}
By \eqref{r-r-derivative-expression-positive3}, \eqref{r''-negative2} and \eqref{r-r-derivative-expression-lower-bd5}, Lemma 2.2 of \cite{H2} and the proof of Theorem 1.1 of \cite{H2}, there exists a constant $\delta_1>0$ such that $r(y)$ can be extended to a function $r(y)\in C^2(\mu,y_1+\delta_1)\cap C([0,y_1+\delta_1))$ which satisfies   
\eqref{r-eqn} in $(\mu,y_1+\delta_1)$ such that \eqref{r-initial-value},  \eqref{r-r-derivative-expression-positive4} holds. By \eqref{r-derivative=0}, $r_y(y_1)=0$. Hence by \eqref{r-eqn},
\begin{equation*}
r_{yy}(y_1)=\left(n-1-\frac{1}{\lambda}\right)\frac{1}{r_1}<0.
\end{equation*}
Then by decreasing the value of $\delta_1$ if necessary we can choose $\delta_1>0$  sufficiently small such that \eqref{r''-negative2-1} holds. Then \eqref{first-derivative-ineqn2} holds and the lemma follows.
\end{proof}

By an argument similar to the proof of Lemma 2.4 of \cite{H2} we have the following result.

\begin{lem}\label{r-monotone-lemma2}
Let $n\ge 2$, $\lambda>0$, $\mu<0$, $r_0>0$, $R_1>0$ and $R_2>0$. Suppose $r(y)\in C^2([R_1,R_2))$ is a solution of \eqref{r-eqn}  in $(R_1,R_2)$ which satisfies 
\begin{equation*}
r(y)-yr_y(y)>0\quad\forall R_1\le y<R_2
\end{equation*} 
and
\begin{equation}\label{r-lower-bd6}
r(y)\ge r_0\quad\forall R_1\le y<R_2.
\end{equation}
Then there exist a constant $\delta_0>0$  such that
\begin{equation}\label{f-structure-ineqn5}
r(y)-yr_y(y)\ge\delta_0\quad\forall R_1<y<R_2.
\end{equation}
\end{lem}

We are now ready for the proof of Theorem \ref{existence-thm}.

\noindent{\bf Proof of Theorem \ref{existence-thm}}: 

\noindent\underline{\bf Existence}: Let $r_1>0$, $y_1>0$ and $\delta_1>0$, be given by Lemma \ref{local-maximal-existence-lem},
 Lemma \ref{y-limit-positive-bd} and Lemma \ref{local-r-existence-extension-lem} respectively. By Lemma 2.6 there exists a maximal interval $(\mu,y_0)$, $y_0\ge y_1+\delta_1$, such that there exists a solution $r(y)\in C^2(\mu,y_0)\cap C([\mu,y_0))$  of \eqref{r-eqn} in $(\mu,y_0)$ which satisfies \eqref{r-initial-value},  \eqref{r''-negative2-1}, \eqref{first-derivative-ineqn2} and
 \begin{equation}\label{r-r-derivative-expression-positive5}
r(y)-yr_y(y)>0\quad\forall \mu<y<y_0.
\end{equation}
We first claim that 
\begin{equation}\label{r-1st-derivative-negative}
r_y(y)<0\quad\forall y_1<y<y_0.
\end{equation}
Suppose the claim \eqref{r-1st-derivative-negative} does not hold. There there exists $z_0\in (y_1,y_0)$ such that
$r_y(z_0)\ge 0$. Hence by \eqref{first-derivative-ineqn2} there exists a maximal interval $(y_1,z_1)$, $y_1<z_1\le z_0$, such that
\begin{align}\label{r-1st-derivative-negative5}
&r_y(y)<0\quad\forall y_1<y<z_1\quad\mbox{ and }\quad r_y(z_1)=0\notag\\
\Rightarrow\quad&r_{yy}(z_1)\ge 0.
\end{align}
On the other hand by \eqref{r-eqn},
\begin{equation*}
r_{yy}(z_1)=\left(n-1-\frac{1}{\lambda}\right)\frac{1}{r(z_1)}<0
\end{equation*} 
which contradicts \eqref{r-1st-derivative-negative5}. Hence no such $z_0$ exists and the claim \eqref{r-1st-derivative-negative} follows. Note that by \eqref{r-1st-derivative-negative}, 
\begin{equation*}
a_1:=\lim_{y\to y_0}r(y)\quad\mbox{ exists and }\quad 0\le a_1<r_1.
\end{equation*}

\noindent $\underline{\text{\bf Claim 1}}$: $\exists y_2\in (y_1,y_0)$ such that $r_{yy}(y_2)=0$ and $r_{yy}(y)<0$ for any $\mu<y<y_2$.

\noindent
Suppose claim 1 is false. Then by \eqref{r''-negative2-1}, 
\begin{equation}\label{r-2nd-derivative-sign11}
r_{yy}(y)<0\quad\forall \mu<y<y_0.
\end{equation}
Thus
\begin{equation}\label{a2-defn}
a_2:=\lim_{y\to y_0}r_y(y)\quad\mbox{ exists }
\end{equation}
and $a_2\in [-\infty,0)$. Let $y_3=y_1+\frac{\delta_1}{2}$. Then by \eqref{first-derivative-ineqn2} and \eqref{r-2nd-derivative-sign11},
\begin{align}
&r_y(y)\le r_y(y_3)<0\qquad\qquad\quad\forall y_3\le y<y_0\notag\\
\Rightarrow\quad&r(y)\le (y-y_3)r_y(y_3)+r(y_3)\quad\forall y_3\le y<y_0.\label{r-ineqn1}
\end{align}
If $y_0>(r(y_3)/|r_y(y_3)|)+y_3$, then by \eqref{r-ineqn1},
\begin{equation*}
r(y)<0\quad\forall (r(y_3)/|r_y(y_3)|)+y_3\le y<y_0
\end{equation*}
and contradiction arises. Hence
\begin{equation*}
y_0\le (r(y_3)/|r_y(y_3)|)+y_3<+\infty.
\end{equation*}
We now divide the proof of claim 1 into 3 cases.

\noindent $\underline{\text{\bf Case 1}}$: $a_1\ge 0$ and $a_2=-\infty$.

\noindent Since
\begin{equation*}
\lim_{y\to y_0}\frac{1+r_y(y)^2}{(r(y)-yr_y(y))r_y(y)}=-\frac{1}{y_0},
\end{equation*}
there exists a constant $y_4\in (y_1,y_0)$ such that
\begin{align}\label{r'-expressive-upper-lower-bd}
&\frac{1+r_y(y)^2}{(r(y)-yr_y(y))r_y(y)}\ge-\frac{2}{y_0}\qquad\quad\forall y_4<y<y_0\notag\\
\Rightarrow\quad&-\frac{1+r_y(y)^2}{\lambda(r(y)-yr_y(y))}\ge\frac{2}{\lambda y_0}r_y(y)\quad\,\forall y_4<y<y_0.
\end{align}
By \eqref{r-eqn}, \eqref{r-1st-derivative-negative} and \eqref{r'-expressive-upper-lower-bd},
\begin{align*}
&\frac{r_{yy}(y)}{1+r_y(y)^2}\ge\frac{2}{\lambda y_0}r_y(y) \qquad\qquad\qquad\forall y_4<y<y_0\notag\\
\Rightarrow\quad&\tan^{-1}(r_y(y))\ge \tan^{-1}(r_y(y_4))-\frac{2r_1}{\lambda y_0}\quad\forall y_4<y<y_0\notag\\
\Rightarrow\quad& a_2\ge\tan\left(\tan^{-1}(r_y(y_4))-\frac{2r_1}{\lambda y_0}\right)>-\infty\quad\quad\mbox{as }y\to y_0
\end{align*}
and contradiction arises. Hence case 1 does not hold.

\noindent $\underline{\text{\bf Case 2}}$: $a_1>0$ and $0>a_2>-\infty$.

\noindent Then by \eqref{r-2nd-derivative-sign11},
\begin{equation}\label{r'-lower-bd7}
0>r_y(y)>a_2\quad\forall y_1<y<y_0.
\end{equation}
By Lemma \ref{r-monotone-lemma2} there exists a constant $\delta_0>0$ such that \eqref{f-structure-ineqn5} holds with $R_1=y_1$ and $R_2=y_0$.
Then by \eqref{f-structure-ineqn5} with $R_1=y_1$, $R_2=y_0$, \eqref{r'-lower-bd7}, Lemma 2.2 of \cite{H2} and the proof of Theorem 1.1 of \cite{H2}, there exists a constant $\delta_2>0$ such that $r(y)$ can be extended to a solution $r(y)\in C^2(\mu,y_0+\delta_2)\cap C([\mu,y_0+\delta_2))$ of \eqref{r-eqn} in $(\mu,y_0+\delta_2)$ such that \eqref{r-initial-value} and  \eqref{r-r-derivative-expression-positive3} holds for any $\mu<y<y_0+\delta_2$. This contradicts the maximality of the interval $(\mu, y_0)$. Hence case 2 does not hold.

\noindent $\underline{\text{\bf Case 3}}$: $a_1=0$ and $0>a_2>-\infty$.

\noindent By \eqref{r-eqn},
\begin{align}
&\frac{1}{1+a_2^2}\lim_{y\to y_0}r(y)r_{yy}(y)=\lim_{y\to y_0}\frac{r(y)r_{yy}(y)}{1+r_y(y)^2}=n-1\label{r-r''-limit12}\\
\Rightarrow\quad&\lim_{y\to y_0}r_{yy}(y)=+\infty\notag
\end{align}
which  contradicts \eqref{r-2nd-derivative-sign11}. Hence case 3 does not hold.

Thus \eqref{r-2nd-derivative-sign11} does not hold. Hence  claim 1 holds. 

Differentiating \eqref{r-eqn}, by \eqref{r-1st-derivative-negative} and claim 1,
\begin{align}\label{r-3rd-derivative-eqn}
&\frac{r_{yyy}(y)}{1+r_y(y)^2}=\frac{2r_yr_{yy}^2}{(1+r_y^2)^2}-\frac{n-1}{r^2}r_y-\frac{2r_yr_{yy}}{\lambda(r-yr_y)}-\frac{y(1+r_y^2)r_{yy}}{\lambda(r-yr_y)^2}\quad\forall y\in (\mu,y_0)\notag\\
\Rightarrow\quad&\frac{r_{yyy}(y_2)}{1+r_y(y_2)^2}=-(n-1)\frac{r_y(y_2)}{r(y_2)^2}>0.
\end{align} 
Hence by \eqref{r-3rd-derivative-eqn} there exists a constant $y_5\in (y_2,y_0)$ such that
\begin{equation}\label{r''-positive2}
r_{yy}(y)>0\quad\forall y_2<y<y_5.
\end{equation}
Let $y_6=\sup\{a\in (y_2,y_0): r_{yy}(y)>0\quad\forall y_2<y<a\}$. Then by \eqref{r''-positive2} $y_6$ exists and $y_5\le y_6\le y_0$. If $y_6<y_0$, then by the definition of $y_6$, $r_{yy}(y_6)=0$ and $r_{yyy}(y_6)\le 0$. On the other hand by the same argument as before,
\begin{equation*}
\frac{r_{yyy}(y_6)}{1+r_y(y_6)^2}=-(n-1)\frac{r_y(y_6)}{r(y_6)^2}>0
\end{equation*}
and contradiction arises. Hence $y_6=y_0$. Thus
\begin{align}
&r_{yy}(y)>0\qquad\qquad\forall y_2<y<y_0\label{r''-positive3}\\
\Rightarrow\quad&r_y(y_2)<r_y(y)<0\quad\forall y_2<y<y_0.\label{r'-uniform-bd}
\end{align}
Hence \eqref{a2-defn} holds and $a_2\in (-\infty,0]$.

\noindent $\underline{\text{\bf Claim 2}}$:  $y_0=+\infty$. 

Suppose claim 2 is false. Then $y_0<+\infty$. By \eqref{r'-uniform-bd} and an argument similar to the proof of case 2 above we get that $a_1=0$. 

Suppose $a_2\in (-\infty,0)$. By \eqref{r-eqn}, \eqref{r-r''-limit12} holds. Then by \eqref{r-r''-limit12} there exists a constant $y_7\in (y_2,y_0)$ such that
\begin{align}
&r(y)r_{yy}(y)\ge c_5\quad\forall y_7<y<y_0\notag\\
\Rightarrow\quad&r_yr_{yy}\le c_5\frac{r_y}{r}\qquad\forall y_7<y<y_0\label{r'-r''-upper-bd3}
\end{align}
where $c_5=(n-1)(1+a_2^2)/2$.
Integrating \eqref{r'-r''-upper-bd3} over $(y_7,y)$,
\begin{align*}
&r_y(y)^2\le r_y(y_7)^2+2c_5\log (r(y)/r(y_7))\quad\forall y_7<y<y_0\notag\\
\Rightarrow\quad&a_2^2\le -\infty\quad\mbox{ as }y\to y_0
\end{align*}
and contradiction arises. Hence $a_2\not\in (-\infty,0)$. Thus $a_2=0$.

By \eqref{r''-positive3} we can divide the proof of claim 2 into three cases.

\noindent $\underline{\text{\bf Case (i)}}$: $\lim_{y\to y_0}r_{yy}(y)=0$.

\noindent Then by \eqref{r-eqn},
\begin{equation*}
0=\lim_{y\to y_0}\frac{r(y)r_{yy}(y)}{1+r_y(y)^2}=n-1-\frac{1}{\lambda}\lim_{y\to y_0}\frac{r(y)}{r(y)-yr_y(y)}=n-1-\frac{1}{\lambda\left(1-y_0\,\underset{\substack{y\to y_0}}{\lim}(r_y(y)/r(y))\right)}.
\end{equation*}
Hence
\begin{equation}\label{r'-r-ratio-limit}
\lim_{y\to y_0}\frac{r_y(y)}{r(y)}=-c_6
\end{equation}
where $c_6=((\lambda (n-1))^{-1}-1)/y_0>0$. By \eqref{r'-r-ratio-limit} there exists a constant $y_8\in (y_2,y_0)$ such that
\begin{align*}
&\frac{r_y(y)}{r(y)}\ge -2c_6\qquad\qquad\qquad\qquad\qquad\forall y_8<y<y_0\\
\Rightarrow\quad&\log r(y)\ge\log r(y_8)-2c_6(y_0-y_8)\quad\,\forall y_8<y<y_0\\
\Rightarrow\quad& a_1\ge r(y_8)e^{-2c_6(y_0-y_8)}>0\qquad\qquad\quad\,\mbox{ as }y\to y_0
\end{align*}
and contradiction arises. Hence case (i) does not hold.

\noindent $\underline{\text{\bf Case (ii)}}$: $\lim_{y\to y_0}r_{yy}(y)=+\infty$.

\noindent
By \eqref{r-eqn} and the l'Hospital rule,
\begin{equation*}
\lim_{y\to y_0}r(y)r_{yy}(y)=n-1-\frac{1}{\lambda}\lim_{y\to y_0}\frac{r(y)}{r(y)-yr_y(y)}=n-1-\frac{1}{\lambda}\lim_{y\to y_0}\frac{r_y(y)}{-yr_{yy}(y)}=n-1.
\end{equation*}
Hence there exists a constant $y_9\in (y_2,y_0)$ such that
\begin{align*}
&r(y)r_{yy}(y)\ge \frac{n-1}{2}\quad\forall y_9<y<y_0\notag\\
\Rightarrow\quad&r_yr_{yy}\le\frac{n-1}{2r}r_y\quad\forall y_9<y<y_0\notag\\
\Rightarrow\quad&r_y(y)^2\le r_y(y_9)^2+(n-1)\log\, (r(y)/r(y_9))\quad\forall y_9<y<y_0\notag\\
\Rightarrow\quad&a_2^2\le -\infty\quad\mbox{ as }y\to y_0
\end{align*}
and contradiction arises. Hence case (ii) does not hold.

\noindent $\underline{\text{\bf Case (iii)}}$: $\exists\mbox{  a sequence } \{z_i\}_{i=1}^{\infty}\subset (y_2,y_0), z_i\to y_0$ as $i\to\infty$, such that 
\begin{equation}\label{r''-limit-positive}
a_3:=\lim_{i\to\infty}r_{yy}(z_i)\in (0,+\infty).
\end{equation}

\noindent
Then by \eqref{r-eqn} and \eqref{r''-limit-positive},
\begin{equation*}
0=\lim_{i\to\infty}\frac{r(z_i)r_{yy}(z_i)}{1+r_y(z_i)^2}=n-1-\frac{1}{\lambda}\lim_{i\to\infty}\frac{r(z_i)}{r(z_i)-z_ir_y(z_i)}=n-1-\frac{1}{\lambda}\lim_{i\to\infty}\frac{r_y(z_i)}{-z_ir_{yy}(z_i)}=n-1
\end{equation*}
and contradiction arises. Hence case (iii) does not hold.

Hence claim 2 holds. Thus there exists a solution $r(y)\in C^2(\mu,+\infty)\cap C([\mu,+\infty))$ of \eqref{r-initial-value-problem} which satisfies \eqref{r-r-derivative-expression-positive}.

\noindent\underline{\bf Uniqueness}: Suppose $r_1(y), r_2(y)\in C^2(\mu,+\infty)\cap C([\mu,+\infty))$ are two solutions of \eqref{r-initial-value-problem} which satisfy
\eqref{r-r-derivative-expression-positive}. Let $y_1(r), y_2(r)$ be the inverse of $r_1(y), r_2(y)$, in a small neighbourhood of $y=\mu$. Let $R_0>0$ by given by Lemma \ref{local-existence-lem}.  Then by decreasing $R_0>0$ if necessary we get that $y_1(r), y_2(r)\in C^1([0,R_0])\cap C^2((0,R_0])$ satisfy \eqref{y-ivp-problem1} and \eqref{y-y-derivative-ineqn}. Hence by Lemma \ref{local-existence-lem},
\begin{align}
y_1(r)=y_2(r)\quad\forall 0<r\le R_0\quad\Rightarrow\quad&r_1(y)=r_2(y)
\qquad\forall\mu\le y\le y_1(R_0)\label{r12=}\\
\Rightarrow\quad&r_{1,y}(y)=r_{2,y}(y)\quad\forall\mu\le y\le y_1(R_0).\label{r12'=}
\end{align}
By \eqref{r-initial-value-problem}, \eqref{r12=}, \eqref{r12'=} and standard ODE theory,
\begin{equation}\label{r12-=near-end}
r_1(y)=r_2(y)\quad\forall y\ge y_1(R_0).
\end{equation} 
By \eqref{r12=} and \eqref{r12-=near-end}, 
\begin{equation*}
r_1(y)=r_2(y)\quad\forall y\ge\mu
\end{equation*} 
and the theorem follows.

{\hfill$\square$\vspace{6pt}}

\noindent{\bf Proof of Theorem \ref{asymptotic-property-thm}}: 
We first observe that by the proof of Theorem \ref{existence-thm} there exist constants $y_2>y_1>0$ such that both \eqref{first-derivative-ineqn} and \eqref{second-derivative-ineqn} holds. So we only need to prove
\eqref{r'-limit-as-y-infty} and \eqref{r-limit-as-y-infty}.
By \eqref{first-derivative-ineqn},
\begin{equation*}
a_1:=\lim_{y\to +\infty}r(y)\quad
\mbox{ exists and }\quad 0\le a_1<+\infty.
\end{equation*}
For any $y>2y_2$, by \eqref{second-derivative-ineqn} and the mean value theorem there exists $\xi\in (y/2,y)$ such that
\begin{equation}\label{yr'-bd}
r(y)-r(y/2)=(z/2)r_y(\xi)<(y/2)r_y(y)<0.
\end{equation}
Letting $y\to\infty$ in \eqref{yr'-bd}, we get \eqref{r'-limit-as-y-infty}.
Suppose $a_1>0$. 
Then by \eqref{r-initial-value-problem} and \eqref{r'-limit-as-y-infty},
\begin{equation}\label{r''-negative-limit}
\lim_{y\to +\infty}r_{yy}(y)=\lim_{y\to +\infty}\frac{r_{yy}(y)}{1+r_y(y)^2}=\left(n-1-\frac{1}{\lambda}\right)\frac{1}{a_1}<0.
\end{equation}
By \eqref{r''-negative-limit} there exists $y_{10}>y_2$ such that
\begin{equation*}
r_{yy}(y)<0\quad\forall y>y_{10}
\end{equation*}
which contradicts \eqref{second-derivative-ineqn}. Hence $a_1=0$ and the theorem follows.

{\hfill$\square$\vspace{6pt}}

Since \eqref{r-eqn} is symmetric with respect to $y$, by Lemma 2.1 of \cite{H2} and an argument similar to proof of Theorem \ref{existence-thm} and Theorem \ref{asymptotic-property-thm} we get the following result which is stated without proof in \cite{DLW}.

\begin{thm}(cf. Theorem 20 of \cite{DLW})
For any $n\ge 2$, $\frac{1}{n}<\lambda<\frac{1}{n-1}$, $r_1>0$, there exists a unique even solution $r(y)\in C^2(-\infty,+\infty)$ of 
\begin{equation*}
\left\{\aligned
&\frac{r_{yy}(y)}{1+r_y(y)^2}=\frac{n-1}{r(y)}-\frac{1+r_y(y)^2}{\lambda(r(y)-yr_y(y))},\quad r(y)>0,\quad\mbox{ in }(-\infty,+\infty)\\
&r(0)=r_1,r_y(0)=0 
\endaligned\right.
\end{equation*}
which satisfies
\begin{equation*}
r(y)-yr_y(y)>0\quad\forall y\in\Rr
\end{equation*}
and
\begin{equation*}
\left\{\aligned
&r_y(y)>0\quad\forall y<0\\
&r_y(y)<0\quad\forall y>0. 
\endaligned\right.
\end{equation*}
Moreover there exists a constant $y_2>0$ such that
\begin{equation*}
\left\{\aligned
&r_{yy}(y)<0\quad\forall |y|<y_2\\
&r_{yy}(\pm y_2)=0\\
&r_{yy}(y)>0\quad\forall |y|>y_2, 
\endaligned\right.
\end{equation*}
\begin{equation*}
\lim_{|y|\to\infty}yr_y(y)=0,\quad \lim_{|y|\to\infty}r(y)=0.
\end{equation*}
\end{thm}


\begin{thebibliography}{99}

\bibitem[BN]{BN} H.~Bray and A.~Neves, {\em Classification of prime $3$-manifolds with $\sigma$-invariant greater than $\mathbb{RP}^3$}, Ann. of Math. 159 (2004), no. 2, 407--424.

\bibitem[CD]{CD} B.~Choi and P.~Daskalopoulos, {\em Evolution of non-compact hypersurfaces by inverse mean curvature flow}, arxiv:1811.04594.

\bibitem[DH]{DH} P.~Daskalopoulos and G.~Huisken, {\em Inverse mean curvature flow of entire graphs}, arxiv:1709.06665v1.

\bibitem[DLW]{DLW} G.~Drugan, H.~Lee and G.~Wheeler, {\em Solitons for the inverse mean curvature flow}, Pacific J. Math. 284 (2016), no. 2, 309--326.

\bibitem[G]{G} C.~Gerhardt, {\em Flow of nonconvex hypersurfaces into spheres}, J. Differential Geom. 32 (1990), no. 1, 299--314.

\bibitem[H1]{H1} K.M.~Hui, {\em Existence of self-similar solutions of the inverse mean curvature flow}, Discrete and Contin. Dynamical Systems 39 (2019), no. 2, 863--880.

\bibitem[H2]{H2} K.M.~Hui, {\em Existence of hypercylinder expanders of the inverse mean curvature flow}, arxiv:1803.07425v3.

\bibitem[HuI1]{HuI1} G.~Huisken and T.~Ilmanen, {\em The Riemannian Penrose inequality},  Internat. Math. Res. Notices  (1997), no. 20, 1045--1058.

\bibitem[HuI2]{HuI2} G. Huisken and T. Ilmanen, {\em A note on inverse mean curvature flow}, in  Proceedings of the workshop on nonlinear partial differential equations, Saitama University, Sept. 1997.

\bibitem[HuI3]{HuI3} G.~Huisken and T.~Ilmanen, {\em The inverse mean curvature flow and the Riemannian Penrose inequality}, J. Differential Geom. 59 (2001), no. 3, 353--437.

\bibitem[HuI4]{HuI4} G.~Huisken and T.~Illmanen, {\em Higher regularity of the inverse mean curvature flow}, J. Differential Geom. 80 (2008), no. 3, 433--451.

\bibitem[KP1]{KP1} D.~Kim and J.~Pyo, {\em Translating solitons for the inverse mean curvature 
flow}, Results Math. 74 (2019), no. 1, article 64.

\bibitem[KP2]{KP2} D.~Kim and J.~Pyo, {\em $O(m)\times O(n)$-Invariant homothetic solitons for inverse
mean curvature flow in $\Rr^{m+n}$}, Nonlinearity 32 (2019), no. 10, 3873--3911.

\bibitem[S]{S} K.~Smoczyk, {\em Remarks on the inverse mean curvature flow}, Asian J. Math. 4 (2000), no. 2, 331--335. 

\bibitem[U]{U} J.~Urbas, {\em On the expansion of starshaped hypersurfaces by symmetric functions of their principle curvatures}, Math. Z. 205 (1990), no. 3, 355--372.

\end{thebibliography}
\end{document}